\documentclass[11pt]{article}

\usepackage[dvips,a4paper,hmargin=2.75cm,vmargin=3cm]{geometry}
\usepackage[utf8]{inputenc}
\usepackage{hyperref}
\usepackage{amssymb,amsmath,amsthm,mathtools}
\usepackage{enumitem}
\usepackage{cite}
\usepackage[mathscr]{euscript}
\usepackage[T1]{fontenc}
\usepackage{lmodern}
\usepackage{fontawesome}
\usepackage{microtype}
\usepackage{graphicx}
\usepackage[normalem]{ulem}
\usepackage{authblk}
\usepackage[small]{titlesec}
\usepackage{cite}

\hypersetup{colorlinks=true,linkcolor=blue,citecolor=blue,
			pdfinfo={
			Title   = {A sufficient condition for oversampling on symmetric regular
						de Branges spaces},
			Author  = {Luis O. Silva, Julio H. Toloza},
			Subject = {Manuscript}}
			}

\titleformat*{\subsection}{\bfseries\itshape}

\providecommand{\keywords}[1]
{
\noindent{\small
	\textbf{Keywords:} #1}
}

\providecommand{\msc}[1]
{
\noindent{\small
	\textbf{2020 MSC:} #1}
}

\newtheorem{theorem}{Theorem}[section]
\newtheorem{definition}[theorem]{Definition}
\newtheorem{lemma}[theorem]{Lemma}
\newtheorem{corollary}[theorem]{Corollary}
\newtheorem{proposition}[theorem]{Proposition}

\theoremstyle{definition}

\newtheorem{examplex}[theorem]{Remark}
\newenvironment{remark}
{\pushQED{\qed}\examplex}
	{\popQED\endexamplex}

\numberwithin{equation}{section}

\newcommand{\R}{{\mathbb R}}
\newcommand{\N}{{\mathbb N}}
\newcommand{\Z}{{\mathbb Z}}
\newcommand{\C}{{\mathbb C}}
\newcommand{\K}{{\mathbb K}}

\newcommand{\cA}{\mathcal{A}}
\newcommand{\cB}{\mathcal{B}}
\newcommand{\cC}{\mathcal{C}}
\newcommand{\cD}{\mathcal{D}}
\newcommand{\cH}{\mathcal{H}}
\newcommand{\cO}{\mathcal{O}}

\newcommand{\cJ}{\mathcal{J}}
\newcommand{\cL}{\mathcal{L}}

\newcommand{\cPW}{\mathcal{PW}}

\DeclareMathOperator{\tr}{tr}
\DeclareMathOperator{\lspan}{span}
\DeclareMathOperator{\ai}{Ai}
\DeclareMathOperator{\bi}{Bi}
\DeclareMathOperator{\wi}{Si}

\DeclareMathOperator{\dom}{\cD}

\DeclareMathOperator{\ran}{ran}
\DeclareMathOperator{\sub}{Sub}
\DeclareMathOperator{\assoc}{Assoc}

\newcommand{\spec}{\sigma}

\newcommand{\ds}[1]{\displaystyle{#1}}
\newcommand{\abs}[1]{\left\lvert #1 \right\rvert}
\newcommand{\abss}[1]{\bigl\lvert #1 \bigr\rvert}
\newcommand{\norm}[1]{\left\lVert #1 \right\rVert}
\newcommand{\inner}[2]{\left\langle#1,#2\right\rangle}

\newcommand{\cc}[1]{\overline{#1}}

\newcommand{\defeq}{\mathrel{\mathop:}=}

\title {\bf Oversampling on a class of symmetric regular
	de Branges spaces}

\author[1]{Luis O. Silva%
			\thanks{\faEnvelopeO\, {silva@iimas.unam.mx}}%
			}
\author[2]{Julio H. Toloza%
			\thanks{\faEnvelopeO\, {julio.toloza@uns.edu.ar (corresponding author)}}%
			}

\affil[1]{Departamento de Física Matemática\\
		 Instituto de Investigaciones en Matemáticas Aplicadas y en Sistemas\\
		 Universidad Nacional Autónoma de México\\
		 C.P. 04510, México D.F.}
\affil[2]{Instituto de Matemática (INMABB)\\
		 Departamento de Matemática\\
		 Universidad Nacional del Sur (UNS) - CONICET\\
		 Bahía Blanca, Argentina}

\date{}

\begin{document}

\maketitle

\begin{abstract}
A de Branges space $\cB$ is regular if the constants belong to its space of associated
functions and is symmetric if it is isometrically invariant under the map
$F(z) \mapsto F(-z)$. Let $K_\cB(z,w)$ be the reproducing kernel in $\cB$ and $S_\cB$ be
the operator of multiplication by the independent variable with maximal domain in
$\cB$.

Loosely speaking, we say that $\cB$
has the $\ell_p$-oversampling property relative to a proper subspace
$\cA$ of it, with $p\in(2,\infty]$, if there exists
$J_{\cA\cB}:\C\times\C\to\C$ such that $J(\cdot,w)\in\cB$ for all $w\in\C$,
\begin{equation*}
\sum_{\lambda\in\sigma(S_{\cB}^{\gamma})}
\left(\frac{\abs{J_{\cA\cB}(z,\lambda)}}{K_\cB(\lambda,\lambda)^{1/2}}\right)^{p/(p-1)}
<\infty
\quad\text{and}\quad
F(z) = \sum_{\lambda\in\sigma(S_{\cB}^{\gamma})}
		\frac{J_{\cA\cB}(z,\lambda)}{K_\cB(\lambda,\lambda)}F(\lambda),
\end{equation*}
for all $F\in\cA$ and almost every self-adjoint extension $S_{\cB}^{\gamma}$ of $S_\cB$. This
definition
is motivated by the well-known oversampling property of Paley-Wiener spaces.

In this paper we provide sufficient conditions for a symmetric,
regular de Branges space
to have the $\ell_p$-oversampling property relative to a chain of de Branges subspaces of it.
\end{abstract}

\bigskip

\msc{
42C15, 
46E22, 
47A06, 
47B32, 
94A20 
}


\section{Introduction}

\subsection{de Branges spaces}

A reproducing kernel Hilbert space of entire functions is a de Branges space $\cB$ if
it is isometrically invariant under the maps
\[
F(z) \mapsto F^\#(z) = F(z^*)^*\quad\text{and}\quad
F(z) \mapsto \frac{z-w^*}{z-w}F(z),
\]
where $w\in\C$ denotes a zero, if any, of the function $F$. Additionally, $\cB$
is regular (or short in some literature \cite{dym-mckean}) if it is closed under the map
\[
F(z)\mapsto \frac{F(z) - F(w)}{z-w}
\]
for some (hence every) fixed $w\in\C\setminus\R$; equivalently, $\cB$ is regular if
the constants belong to the space of associated functions $\assoc(\cB)$ \cite{debranges-book}.
Finally, $\cB$ is symmetric when it is isometrically invariant under the map
\[
F(z) \mapsto F^{-}(z) = F(-z).
\]

The reproducing kernel of a de Branges space $\cB$ will be denoted by $K(z,w)$
(or $K_\cB(z,w)$ if disambiguation is necessary) with the convention
$K(\cdot,w)\in\cB$ ($w\in\C$) so
\[
F(w) = \inner{K(\cdot,w)}{F}_{\cB} \quad (F\in\cB).
\]
Moreover, $K(z,w) = K(w,z)^* = K(w^*,z^*)$.
The inner product on $\cB$ ---or any other Hilbert space--- will be assumed linear in
its second argument.

A de Branges space can be generated from an Hermite-Biehler (or de
Branges) function, that is, an entire function $E$ that obeys the
inequality $\abs{E(z)}>\abs{E(z^*)}$ for all $z\in\C_+$ (the upper
open half-plane). Given such a function $E$, the corresponding de
Branges space is a subspace of the Hardy space $\cH^2(\C_+)$, to wit,
\[
\cB(E) \defeq \left\{F \text{ entire}: F/E , F^\#/E \in \cH^2(\C_+)\right\}
\]
endowed with the inner product of $\cH^2(\C_+)$. According to \cite[Thm.~23]{debranges-book},
given a (non-trivial) de Branges space $\cB$ there exists an Hermite-Biehler function
$E$ such that $\cB = \cB(E)$ isometrically.
If $\cB$ is regular, then $E$ is zero-free on the real line \cite[\S 6.2, Prop.~2]{dym-mckean}.
By \cite[\S 6.1, Prop.~4]{dym-mckean}, $\cB$ is symmetric if and only if $E$ satisfies
$E^\#(z) = E(-z)$. As a consequence, if $\cB$ is regular and symmetric
then $\cB=\cB(E)$ isometrically for a unique zero-free Hermite-Biehler function $E$ such
that $E(0) = 1$ (see \cite[\S 6.1, Prop.~5]{dym-mckean}).

A de Branges subspace of $\cB$ is a (closed) subspace $\cA\subset\cB$ that itself is a
de Branges space (endowed with the inner product of $\cB$). The collection of de Branges
subspaces of a given de Branges space is totally ordered by inclusion: If $\cA_1$ and $\cA_2$
are de Branges subspaces of $\cB$, then either $\cA_1\subset\cA_2$ or $\cA_2\subset\cA_1$
\cite[Thm.~35]{debranges-book}; see also \cite[Thm.~2.13]{woracek-18}. By definition,
\[
\sub(\cB) \defeq \left\{\cA\subset\cB : \cA \text{ is a de Branges subspace}\right\}
\]
is the maximal chain of (nested) de Branges subspaces of $\cB$. Interestingly, all the elements
in $\sub(\cB)$ are regular if and only if one element is \cite[Remark~2.16]{woracek-18}.

Functions of a regular de Branges space are of exponential type, that is,
\[
\tau(F) \defeq \limsup_{R\to\infty} \frac{\max_{\abs{z}=R}\log\abs{F(z)}}{R} < \infty
	\quad (F\in\cB).
\]
Moreover, the exponential type is uniformly bounded across $\cB=\cB(E)$. Namely,
\[
\tau(\cB) = \sup_{F\in\cB} \tau(F) = \tau(E) < \infty;
\]
see \cite[Thm.~3.4]{kaltwor-05} (also \cite[\S 6.2, Prop.~2]{dym-mckean}). $\cB$ is of normal
type if $\tau(\cB)>0$ or minimal type if $\tau(\cB)=0$.

Of upmost importance in the theory of de Branges spaces is the operator of multiplication
by the independent variable $S:\cD(S)\to\cB$ (sometimes denoted $S_\cB$), that is,
\[
(SF)(z) \defeq zF(z),\quad \dom(S) \defeq \left\{F(z)\in\cB : zF(z)\in\cB\right\}.
\]
This operator is symmetric with deficiency indices $(1,1)$ so it has a one-parameter
family of (canonical) self-adjoint extensions $S^\gamma$, where for convenience we
assume $\gamma\in[0,\pi)$. The domain of $S$ is not necessarily dense in $\cB$ although
its codimension never exceeds 1. This implies that the self-adjoint extensions of $S$
are all operators with at most one possible exception (and this happens if and only if
$\dim{\dom(S)^\perp}=1$),
in which case the exceptional extension is a proper linear relation.
The operator $S$ is also regular\footnote{That is, given $z\in\C$, there exists
$c_z>0$ such that $\norm{(S-zI)F}\ge c_z\norm{F}$  for all $F\in\cD(S)$.}
\cite[Cor.~4.7]{kaltwor-99}, which implies that any spectrum $\spec(S^\gamma)$
consists of isolated eigenvalues of multiplicity equal to one and
\[
\bigcup_{\gamma\in[0,\pi)}\spec(S^\gamma) = \R,\quad
\spec(S^\beta)\cap\spec(S^{\gamma}) = \emptyset\quad (\beta\ne\gamma).
\]
A more detailed recount of these facts (along with references) is in \cite{us-handbook}.

\begin{remark}
\label{rem:why-sampling}
Suppose $S_\cB^\gamma$ is not the exceptional extension of $S_\cB$ (in the event
that $\cD(S_\cB)$ being not densely defined).
Noting that $K_\cB(z,\lambda)$ is (up to a constant factor) the eigenfunction associated with
$\lambda\in\spec(S^\gamma_\cB)$, one readily obtains the so-called sampling formula
\begin{equation}
\label{eq:sampling-formula}
F(z) = \sum_{\lambda\in\sigma(S^\gamma_\cB)}
	\frac{K_\cB(z,\lambda)}{K_\cB(\lambda,\lambda)}F(\lambda)\quad (F\in\cB).
\end{equation}
This formula in fact establishes a bijective correspondence between $\cB$ and the
space of sequences $\ell_2(\cB,\gamma)$, where
\begin{equation*}
\ell_p(\cB,\gamma)
	\defeq \left\{\{\alpha_\lambda\}_{\lambda\in\sigma(S^{\gamma}_\cB)}\subset\C :
		\left\{\frac{\alpha_\lambda}
		{K_\cB(\lambda,\lambda)^{1/2}}\right\}_{\lambda\in\sigma(S^{\gamma}_\cB)}
		\in\ell_p\right\}.
\end{equation*}
One can also say that \eqref{eq:sampling-formula} is stable under weighted $\ell_2$
perturbations of the samples. In passing, we note that \eqref{eq:sampling-formula} converges
in the norm of $\cB$, which implies uniform convergence in compact subsets of $\C$.
\end{remark}

\subsection{Oversampling}

\begin{definition}
\label{def:oversampling}
Given $p\in(2,\infty]$,
a de Branges space $\cB$ has the $\ell_p$-oversampling property relative to $\cA\in\sub(\cB)$
if there exists a function $J_{\cA\cB}:\C\times\C\to\C$ having the following properties:
\begin{enumerate}[label={\bf (o\arabic*)}, ref={(o\arabic*)}]
\item	\label{it:o1}
	For all $w\in\C$, $J_{\cA\cB}(\cdot,w)\in\cB$. Moreover,
	\begin{equation*}
	J_{\cA\cB}(z,w)	=J_{\cA\cB}(w,z)^*
					= J_{\cA\cB}(w^*,z^*).
	\end{equation*}

\item	\label{it:o2}
	The map
	$\ds{\cB\ni F \mapsto \inner{J_{\cA\cB}(\cdot,z)}{F(\cdot)}_\cB\in\text{\rm Hol}(\C)}$
	defines an operator in $\cL(\cB)$. Moreover,
	\[
	\inner{J_{\cA\cB}(\cdot,z)}{F(\cdot)}_\cB = F(z)
	\]
	for all $F\in\cA$.

\item There exists a self-adjoint extension $S_{\cB}^{\gamma}$ of $S_{\cB}$ such that
	\begin{equation*}
	\sum_{\lambda\in\sigma(S_{\cB}^{\gamma})}
	\left(\frac{\abs{J_{\cA\cB}(z,\lambda)}}{K_\cB(\lambda,\lambda)^{1/2}}\right)^{p/(p-1)}
	\end{equation*}
	converges uniformly in compact subsets of $\C$.
	\label{it:o3}
\end{enumerate}
\end{definition}

The second part of \ref{it:o2} is equivalent to the identity
\begin{equation*}
\label{eq:oversampling-general}
F(z) = \sum_{\lambda\in\sigma(S_{\cB}^{\gamma})}
		\frac{J_{\cA\cB}(z,\lambda)}{K_\cB(\lambda,\lambda)}F(\lambda),
\end{equation*}
for every $F\in\cA$ and every self-adjoint extension $S_{\cB}^{\gamma}$ of $S_\cB$
(barring the exceptional one, should $S_\cB$ be not densely defined). The convergence
is in the norm of $\cB$, hence uniform in compact subsets of $\C$.
There are many functions having the properties
\ref{it:o1} and \ref{it:o2}, for instance $K_\cB(z,w)$, $K_\cA(z,w)$ and their convex
linear combinations. What makes oversampling interesting is \ref{it:o3}. Given
$\epsilon\in\ell_p(\cB,\gamma)$ and $F\in\cA$, define
\begin{equation*}
F_\epsilon(z) \defeq
	\sum_{\lambda\in\sigma(S_{\cB}^{\gamma})}
		\frac{J_{\cA\cB}(z,\lambda)}{K_\cB(\lambda,\lambda)}
		\left(F(\lambda)+\epsilon_\lambda\right).
\end{equation*}
Then, for every compact $\K\subset\C$, there exists $C(\K)>0$ such that
\begin{equation*}
\abs{F(z) - F_\epsilon(z)}
	\le C(\K)\norm{\epsilon}_{\ell_p(\cB,\gamma)}\quad
	(z\in\K).
\end{equation*}
In other words, we obtain a sampling formula for elements in $\cA$, using
sampling points associated with $\cB$, that is stable under weighted
$\ell_p$ ---rather than $\ell_2$--- perturbations of the samples
(cf. Remark~\ref{rem:why-sampling}).

The main results of this paper are Corollary~\ref{cor:the-corollary} and 
Theorem~\ref{thm:the-theorem}, which can be summarized (in a somewhat weaker form)
as follows:

\begin{theorem}
\label{thm:oversampling}
Let $\cB$ be a symmetric, regular de Branges space. Let $\cB_l\subset\cB$ be the
de Branges subspace associated to $l\in(0,b]$ (see Theorem~\ref{thm:big-theorem}).
Denote $\tau(l) = \tau(\cB_l)$. Suppose there
exists a closed interval $I\subset (0,b]$ such that $\tau$ is strictly increasing
on $I$ and $\tau'$ is absolutely continuous on $I$. Then,
for every $a$ in the interior of $I$,
$\cB=\cB_b$ has the $\ell_p$-oversampling property
relative to every $\cB_l\in\sub(\cB_a)$ for every $p\in(2,\infty)$. Moreover,
if $\tau$ is a linear function on $I$, then the previous statement also
holds true for $p=\infty$.
\end{theorem}

\subsection{Background}

The notion of oversampling stems from the theory of Paley-Wiener spaces
\cite{partington-book,zayed-93}, that is, the spaces of Fourier
transforms of functions with given compact support centered at zero,
\begin{equation*}
	\cPW_a \defeq \left\{F(z)=\int_{-a}^a e^{-ixz}\phi(x)dx : \phi\in
	L^2(-a,a)\right\}.
\end{equation*}
The linear set $\cPW_a$, equipped with the norm $\norm{F}=\norm{\phi}$, is
a de Branges space generated by the Hermite-Biehler
function $E_a(z)=\exp(-iza)$.
In this setting, the sampling formula \eqref{eq:sampling-formula} is known as the
Whittaker-Shannon-Kotelnikov theorem and takes the form
\begin{equation}
\label{eq:wsk-sampling}
	F(z) =\sum_{n\in\Z}
		\mathcal{G}_a\left(z,\tfrac{n\pi}{a}\right)F\left(\tfrac{n\pi}{a}\right),
	\qquad
	\mathcal{G}_a\left(z,w\right)
		\defeq\frac{\sin\left[a(z-\cc{w})\right]}{a(z-\cc{w})}.
\end{equation}
The function $\mathcal{G}_a\left(z,w\right)$ is referred to as the sampling
kernel, while the separation between sampling points $\pi/a$ is known as the
Nyquist rate.

As shown in
\cite[Thm. 7.2.5]{partington-book}, every $F\in\cPW_a\subset\cPW_b$ ($a<b$) admits
the representation
\begin{equation}
  \label{eq:wsk-oversampling}
  	F(z)=\sum_{n\in\Z}\mathcal{G}_{ab}\left(z,\tfrac{n\pi}{b}\right)
	F\left(\tfrac{n\pi}{b}\right),
\end{equation}
with the modified sampling kernel
\begin{equation}
\label{eq:oversampling-PW}
\mathcal{G}_{ab}(z,w)
 \defeq \frac{2}{b-a}\frac{\cos((z-\cc{w})a) - \cos((z-\cc{w})b)}{(z-\cc{w})^2}.
\end{equation}
The convergence of \eqref{eq:wsk-sampling} is clearly unaffected by
$\ell_2$-perturbations of the samples $F\left(\frac{n\pi}{a}\right)$. However,
the oversampling formula \eqref{eq:wsk-oversampling} is more stable in
the sense that it converges under $\ell_\infty$-perturbations of the samples. That is, if the
sequence $\{\delta_n\}_{n\in\Z}$ is bounded and one defines
\begin{equation*}
	F_\delta(z)\defeq
	\sum_{n\in\Z}\mathcal{G}_{ab}\left(z,\tfrac{n\pi}{b}\right)
	\left[F\left(\tfrac{n\pi}{b}\right)+\delta_n\right],
\end{equation*}
then $\abss{F_\delta(z)-F(z)}$ is uniformly bounded in compact subsets of $\C$. In fact,
using methods of Fourier analysis, one can prove that \eqref{eq:wsk-oversampling} is
uniformly bounded on the real line.
In the jargon of signal analysis, a more stable sampling formula is obtained at the expense of
collecting samples at a higher Nyquist rate.

Generalizations of the sampling formula \eqref{eq:wsk-sampling} have been a subject
of research mainly within the theory of reproducing
kernel Hilbert spaces \cite{butzer,garcia-13,higgins1,jorgensen-2016};
a classical result is the Kramer's sampling theorem.
The particular case of reproducing kernel Hilbert spaces related
to the Bessel-Hankel transform has been studied by Higgins \cite{higgins0}.
Sampling theorems associated with Sturm-Liouville problems have been discussed in
\cite{zayed1,zayed2,zayed3}. See also \cite{zayed4}.

To the best of our knowledge, oversampling on de Branges spaces besides the Paley-Wiener
class has not been discussed until recently, where this subject has been
touched upon for (regular) de Branges spaces associated with regular Schrödinger
operators \cite{siltoluri-19}. Oversampling on a class of non-regular de Branges spaces,
associated with perturbed Bessel operators on a finite interval, has been shown
in \cite{toluri-20}.

\begin{remark}
By the time we gave the final touches to the present paper, we learned of a recent work
by Baranov et al. \cite{baranov-23}, where a different notion of oversampling 
is discussed for a class of so-called model spaces, among which the de
Branges spaces form a subclass. It is worth remarking that
Definition~\ref{def:oversampling} establishes oversampling as a
geometric property of de Branges spaces. Also, unlike the approach in
\cite{baranov-23}, ours is based on the relation of this kind of
spaces to canonical systems.
\end{remark}

\subsection{Paper's layout}

In Section~\ref{sec:general} we characterize the oversampling property in terms of the
existence of certain bounded operators.
Section~\ref{sec:prelim} summarizes the connection between regular canonical systems and
regular de Branges spaces. The main results of this paper are elaborated
in Section~\ref{sec:main}. Finally, Section~\ref{sec:example} ---complemented with 
the Appendix--- is dedicated to an example of
a symmetric regular de Branges space, having the $\ell_\infty$-oversampling property
but outside the scope of Theorem~\ref{thm:oversampling}.


\section{Some general statements}
\label{sec:general}

The $\ell_p$-oversampling property (of a de Branges space $\cB$ relative to a de Branges
subspace $\cA$) can be portrayed as 
the existence of a suitable deformation of the orthogonal projector onto $\cB\ominus\cA$.

\begin{proposition}
A function $J_{\cA\cB}:\C\times\C\to\C$ obeys \ref{it:o1} and \ref{it:o2}
if and only if
\begin{equation}
\label{eq:J-Omega}
J_{\cA\cB}(z,w) = \inner{K_\cB(\cdot,z)}{(P_{\cA\cB} + \varOmega_{\cA\cB})K_\cB(\cdot,w)}_\cB,
\end{equation}
where $P_{\cA\cB}:\cB\to\cB$ is the orthogonal projector onto $\cA$ and
$\varOmega_{\cA\cB}:\cB\to\cB$ is a bounded, self-adjoint operator such that
\begin{equation*}
\ker(\varOmega_{\cA\cB}) = \ran(P_{\cA\cB}),
\quad \ran(\varOmega_{\cA\cB})\subset \ran(I-P_{\cA\cB})
\end{equation*}
and
\begin{equation}
\label{eq:real-operator}
(\varOmega_{\cA\cB}F)^\#(z) = (\varOmega_{\cA\cB}F^\#)(z) \quad (F\in\cB).
\end{equation}
\end{proposition}
\begin{proof}
($\impliedby$) Let $\Omega_{\cA\cB}$ be an operator as stated and define $J_{\cA\cB}$
according to \eqref{eq:J-Omega}. Clearly \ref{it:o1} holds true. Note that
\[
J_{\cA\cB}(z,w) = K_{\cA}(z,w) + J_2(z,w)
\]
with $J_2(z,w) = \inner{K_\cB(\cdot,z)}{\varOmega_{\cA\cB} K_\cB(\cdot,w)}_\cB$ so
$\inner{J_2(\cdot,z)}{F}_\cB = \inner{K_\cB(\cdot,z)}{\varOmega_{\cA\cB}F}_\cB$.
In other words,
\[
\inner{J_{\cA\cB}(\cdot,z)}{F(\cdot)}_\cB = ((P_{\cA\cB} + \varOmega_{\cA\cB}) F)(z),
\]
which implies \ref{it:o2}.

($\implies$) Due to \ref{it:o1} and the projection theorem, one has
\[
J_{\cA\cB}(z,w) = J_{1}(z,w) + J_2(z,w)
\]
with $J_{1}(\cdot,w)\in\cA$ and $J_{2}(\cdot,w)\in\cB\ominus\cA$. Using \ref{it:o2},
one readily obtains $J_{1}(z,w) = K_{\cA}(z,w)$. Finally, define $\varOmega_{\cA\cB}$ by the
rule
\[
(\varOmega_{\cA\cB}F)(z) = \inner{J_{2}(\cdot,z)}{F}_\cB, \quad F\in\cB\ominus\cA.
\]
Clearly $\varOmega_{\cA\cB}$ is a bounded operator. Due to \ref{it:o1}, it
is self-adjoint and satisfies \eqref{eq:real-operator}.
\end{proof}

In view of the preceding result, we can say that $\cB$ has the $\ell_p$-oversampling property
relative to $\cA$ if and only if there exists a real (in the sense \eqref{eq:real-operator}),
self-adjoint operator $\varOmega_{\cA\cB}\in\cL(\cB\ominus\cA)$ such that
\begin{equation*}
\sum_{\lambda\in\sigma(S_{\cB}^{\gamma})}\left(
	\frac{\abs{\inner{K_\cB(\cdot,z)}{(P_{\cA\cB}+\varOmega_{\cA\cB})K_\cB(\cdot,\lambda)}_\cB}}
	{K_\cB(\lambda,\lambda)^{1/2}}\right)^{p/(p-1)}
\end{equation*}
converges uniformly in compact subsets of $\C$, for all self-adjoint extensions
$S_{\cB}^{\gamma}$ of $S_\cB$ (with possibly one exception).

\begin{remark}
It will be useful to keep in mind the following straightforward facts:
\begin{enumerate}
\item \label{easy-0}
	Let $U$ be a unitary transformation between de Branges spaces $\cB$ and $\cB'=U\cB$.
	Suppose $\cB$ has the $\ell_p$-oversampling property relative to $\cA\in\sub(\cB)$.
	Then $\cB'$ has the $\ell_p$-oversampling property relative to $\cA'=U\cA$.

\item \label{easy-2}
	Let $\cA,\cB\in\sub(\cC)$ with $\cA\subset\cB$. Suppose $\cC$ has the
	$\ell_p$-oversampling property relative to $\cB$. Then $\cC$ has the $\ell_p$-oversampling
	property relative to $\cA$. \qedhere
\end{enumerate}
\end{remark}


\section{Canonical systems and de Branges spaces}
\label{sec:prelim}

\subsection{Canonical systems in a nutshell}

The following is taken mostly from Remling's book \cite{remling-book}.
Given $l\in(0,\infty)$, consider a differential equation of the form
\begin{equation}
\label{eq:canonical-system-base}
\cJ y'(x) = -z H(x)y(x),\quad
\cJ = \begin{pmatrix}
		0 & -1 \\[1mm] 1 & 0
	\end{pmatrix},\quad
x\in(0,l),\quad z\in\C,
\end{equation}
where $H:(0,l)\to\R^{2\times 2}$ is an integrable, essentially non-zero,
non-negative (hence symmetric) matrix-valued function. That is, for (Lebesgue) almost every
$x\in (0,l)$,
\[
H(x) =  \begin{pmatrix}
		h_1(x) & h_3(x) \\[1mm]
		h_3(x) & h_2(x)
		\end{pmatrix}\!,
\quad H(x)\ne 0,
\quad v^*H(x)v\ge 0\text{ for all } v\in\C^{2\times 2},
\]
and $\int_0^l \abs{h_i(x)} dx < \infty$. Without loss of generality,
we can assume that $\tr H(x) = 1$ a.e. $x\in(0,l)$  \cite{winkler-12}.
The class of $2\times 2$ matrix-valued functions having all the properties listed above
will be called $\cC(l)$.

A maximal open interval $I\subset[0,l]$ is singular (or $H$-indivisible)  if
\[
H(x) = \xi_\phi \xi_\phi^t
	\quad \text{a.e. } x\in I,
	\quad \xi_\phi = \begin{pmatrix}
						\cos\phi \\[1mm] \sin\phi
					 \end{pmatrix},
\]
where the constant $\phi\in[0,\pi)$ is the type of the singular interval.

Let $L^2_H(l) = L^2_H(0,l)$ be the linear set of equivalence classes of
measurable vector-valued functions $f:(0,l)\to\R^{2}$ such that
\[
\int_0^l f(x)^*H(x)f(x)dx < \infty;
\]
$L^2_H(l)$ becomes a Hilbert space once endowed with the inner
product
\[
\inner{f}{g}_{L^2_H(l)} \defeq \int_0^l f(x)^*H(x)g(x)dx.
\]

Let $A_l$ be the following linear relation associated with \eqref{eq:canonical-system-base},
\[
A_l \defeq \left\{\begin{gathered}
		\{f,g\}\in L^2_H(l)\oplus L^2_H(l): f \text{ has a representative } \mathring{f}
		\in \text{AC}[0,l] \text{ such that}
		\\[1mm]
		\cJ \mathring{f}'(x) = - H(x)g(x)\text{ a.e. }x\in(0,l) \text{ and }
		\mathring{f}_2(0) = 0
		\end{gathered}\right\}.
\]
The representative $\mathring{f}$ in the definition above is unique if $(0,l)$ does
not consists of a single singular interval \cite[Lem.~2.1]{remling-book}.

Define
\[
A_{l,0} \defeq \left\{\{f,g\}\in A_l: \mathring{f}(l)=0\right\}.
\]
This relation is closed and symmetric with deficiency indices $(1,1)$.
Its self-adjoint extensions are
\[
A^\gamma_{l} = \left\{\{f,g\}\in A_l :
				\mathring{f}_1(l)\sin\gamma -\mathring{f}_2(l)\cos\gamma = 0\right\},
\quad \gamma\in[0,\pi).
\]

\subsection{Associated de Branges spaces}

Let $u(z,x)$ be the solution to the initial-value problem given
by \eqref{eq:canonical-system-base} with initial condition
$u(z,0)=(\begin{smallmatrix}1\\0\end{smallmatrix})$ on a given interval $(0,b)$
which we henceforth assume that it is not a single singular interval, hence $u(z,x)$
is unique.
Also, $u(z,\cdot)\in\ker(A_l-z)$ for every $l\in(0,b]$ and $u(\cdot,x)$ is an entire
function such that $u(z^*,x)^* = u(z,x)^t$.

As shown in \cite[Thm.~4.6]{remling-book},
\[
E_l(z) \defeq u_1(z,l) - i u_2(z,l)
\]
is an Hermite-Biehler function for every $l\in(0,b]$. The corresponding de Branges
space $\cB(E_l)$ has reproducing kernel
\begin{equation}
\label{eq:reproducing-kernel}
K_l(z,w)
	= \int_0^l u(z,x)^t H(x) u(w^*,x) dx
	=  \frac{u(z,l)^t\cJ u(w^*,l)}{z-w^*} \quad (w\ne z^*).
\end{equation}
Moreover, $\cB(E_l) = \cB_l$ isometrically (see \cite[Thm.~4.7]{remling-book}), where
\[
\cB_l \defeq \left\{F(z) = \int_0^l u(z,x)^t H(x) f(x) dx : f\in\cc{\dom(A_{l})}\right\},\quad
	\norm{F}_{\cB_l} \defeq \norm{f}_{L^2_H(l)}.
\]

Some structural properties of regular de Branges spaces are related to properties
of a canonical system associated with it. The following theorem collects relevant
results about this relation; except for minor modifications these statements are taken from
Theorems~4.8, 4.26, 5.21 and 6.11 of \cite{remling-book}.

\begin{theorem}
\label{thm:big-theorem}
Let $\cB=\cB(E)$ be a regular de Branges space with $E(0)=1$.
\begin{enumerate}[label={(\roman*)}]
\item There exists $b\in(0,\infty)$ and $H\in\cC(b)$ such that $E(z)=E_b(z)$.
	The pair $\{b,H\}$ is unique modulo an initial singular interval of type $\pi/2$.

\item In addition, $\cB$ is symmetric if and only if $H(x)$ is diagonal for a.e. $x\in(0,b)$.

\item Suppose $(0,b]$ does not begin with a singular interval of type $\pi/2$. Then,
	\[
	\sub(\cB)=\{\cB(E_l) : l\in(0,b] \text{ is a regular point of } H\}.
	\]

\item For each $l\in(0,b]$, $E_l$ ---hence $\cB(E_l)$--- is of exponential type
	with type
	\[
	\tau(l) = \int_0^l \sqrt{\det H(x)}dx.
	\]
\end{enumerate}
\end{theorem}

\begin{remark}
\begin{enumerate}
\item If $\cB$ is symmetric, then so is every subspace in $\sub(\cB)$.

\item From now on, we write $\cB_b\defeq\cB$, $S_l\defeq S_{\cB_l}$, 
	and $K_l(z,w)\defeq K_{\cB_l}(z,w)$.

\item Note that $\tau$ is strictly increasing on an interval $[x_1,x_2]\subset[0,b]$
	if $H(x)$ is diagonal and has no singular intervals within $[x_1,x_2]$.\qedhere
\end{enumerate}
\end{remark}

The correspondence between the linear relation $A_{l,0}$ and the operator of multiplication
$S_l$ in $\cB_l$ is established as follows (see \cite[Thm.~4.13]{remling-book}).

\begin{theorem}
Let $U:L^2_H(l)\to\cB_l$ be the partial isometry defined by the rule
\begin{equation*}
(Uf)(z) \defeq \int_0^l u(z,x)^t H(x) f(x) dx.
\end{equation*}
Then, $S_l = UA_{l,0}U^*$. Moreover, $\dim{\cD(S_l)^\perp}=1$ if and only if
$(0,l)$ ends with a singular interval $(c,l)$ of type $\omega^\perp$ for some
$\omega\in[0,\pi)$. If so, the exceptional extension $S_l^\omega$ is a proper linear relation,
\begin{equation*}
\dom(S_l^\omega) = \dom(S_l),\quad\text{ and }\quad \cD(S_l)^\perp = \lspan(s^\omega_l(z)),
\end{equation*}
where
\begin{equation*}
s^\omega_l(z)
	= -u_1(z,l)\sin\omega + u_2(z,l)\cos\omega
	= \frac{i}{2}\left(e^{i\omega}E_l(z) - e^{-i\omega}E_l^\#(z) \right).
\end{equation*}
\end{theorem}

Let $S_{b}^\gamma$ denote a canonical selfadjoint extension of $S_b$. Clearly,
$S_{b}^\gamma = UA_{b}^\gamma U^*$.  Assuming that $S_{b}^\gamma$ is not exceptional,
one has the identity
\[
F(z)
	= \sum_{\lambda\in\spec(S_{b}^\gamma)}
		\frac{\inner{u(z^*,\cdot)}{u(\lambda,\cdot)}_{L^2_H(b)}}
		{\norm{u(\lambda,\cdot)}^2_{L^2_H(b)}}\inner{u(\lambda,\cdot)}{f}_{L^2_H(b)}
	= \sum_{\lambda\in\spec(S_{b}^\gamma)}\frac{K_b(z,\lambda)}{K_b(\lambda,\lambda)}F(\lambda),
\]
for every $F\in\cB_b$, where $f\in L^2_H(b)$ obeys $F(z)=\inner{u(z^*,\cdot)}{f}_{L^2_H(b)}$.


\section{Main Results}
\label{sec:main}

Let $a\in(0,b)$ be a regular point of $H$, and consider $F\in\cB_a$. Then
$F(z)=\inner{u(z^*,\cdot)}{f}_{L^2_H(a)}$ for some $f\in L^2_H(a)$. Extend
$f$ to a function in $L^2_H(b)$ by setting $f$ equivalent to 0 in $(a,b)$.
Let $\omega:[a,b]\to\R$ be an absolutely continuous function such that $\omega(a) = 1$ and
$\omega(b) = 0$. Define
\[
R_{ab;\omega}(x) \defeq \chi_{[0,a]}(x) + \omega(x)\chi_{(a,b]}(x).
\]
Then,
\[
F(z)
	= \sum_{\lambda\in\spec(S_{b}^{\gamma})}
		\frac{\inner{u(z^*,\cdot)}
		{R_{ab;\omega}(\cdot)u(\lambda,\cdot)}_{L^2_H(b)}}{K_b(\lambda,\lambda)}F(\lambda).
\]
For $\cB_b$ to have the $\ell_p$-oversampling property relative to $\cB_a$ with
$p\in(2,\infty]$, it suffices to show that
\[
\sum_{\lambda\in\spec(S_{b}^\gamma)}\left(\frac{1}{K_b(\lambda,\lambda)^{1/2}}
		\abs{\inner{u(z^*,\cdot)}{R_{ab;\omega}(\cdot)u(\lambda,\cdot)}_{L^2_H(b)}}\right)^{q}
		<\infty,
\]
with $q\in[1,2)$,
uniformly on compact sets of $\C$. From now on, we denote
\[
J_{ab;\omega}(z,w) \defeq \inner{u(z^*,\cdot)}{R_{ab;\omega}(\cdot)u(w^*,\cdot)}_{L^2_H(b)}.
\]

If $w\ne 0$, then
\[
\begin{aligned}
J_{ab;\omega}(z,w)
	=&\, \int_0^a u(z,x)^t H(x) u(w^*,x) dx
		+ \int_a^b u(z,x)^t H(x) u(w^*,x) \omega(x) dx
	\\[1mm]
	=&\, -\frac{1}{w^*}\int_0^a u(z,x)^t \cJ u'(w^*,x) dx
		- \frac{1}{w^*}\int_a^b u(z,x)^t \cJ u'(w^*,x) \omega(x) dx.
\end{aligned}
\]
Since
\[
\int_0^a u(z,x)^t \cJ u'(w^*,x) dx
	= u(z,a)^t \cJ u(w^*,a) - z \int_0^a u(z,x)^t H(x) u(w^*,x) dx
\]
and
\[
\begin{aligned}
\int_a^b \omega(x)u(z,x)^t \cJ u'(w^*,x) dx
	=&\, - u(z,a)^t \cJ u(w^*,a) - \int_a^b u(z,x)^t \cJ u(w^*,x) \omega'(x) dx
	\\[1mm]
	 &\, - z \int_a^b u(z,x)^t H(x) u(w^*,x) \omega(x) dx,
\end{aligned}
\]
one has
\[
(z-w^*)J_{ab;\omega}(z,w)
	= - \int_a^b u(z,x)^t \cJ u(w^*,x) \omega'(x) dx.
\]
Thus, if $z\in \K$ (a compact subset of $\C$) and $N$ is sufficiently large,
\begin{equation}
\label{eq:expression-to-bound}
\sum_{\substack{\lambda\in\sigma(S_{b}^{\gamma}) \\ \abs{\lambda}> N}}
		\left(\frac{\abs{J_{ab;\omega}(z,\lambda)}}{K_b(\lambda,\lambda)^{1/2}}
		\right)^q
	= \sum_{\substack{\lambda\in\sigma(S_{b}^{\gamma}) \\ \abs{\lambda}> N}}
		\left(\frac{\abs{\int_a^b u(z,x)^t \cJ u(\lambda,x) \omega'(x) dx}}
		{\abs{z-\lambda}\norm{u(\lambda,\cdot)}_{L^2_H(b)}}\right)^q,
\end{equation}
so we only have to deal with the convergence of this last expression.

\begin{remark}
It is interesting to see how $J_{ab;\omega}(z,w)$ looks like when it is written in terms of
elements of the corresponding de Branges space. Clearly,
\[
J_{ab;\omega}(z,w)
	 = - \int_a^b K_x(z,w)\omega'(x) dx
\]
as it follows from \eqref{eq:reproducing-kernel}.
\end{remark}

\begin{proposition}
\label{prop:oversampling-restricted-diagonal-system-constant}
Suppose $H$ has an interval of regular points $[x_1,x_2]$ where
\begin{equation}
\label{eq:diagonal-constant-matrix}
H(x) = \begin{pmatrix} \frac12 + g_0 & 0 \\[1mm] 0   & \frac12 - g_0  \end{pmatrix},
	\quad g_0\in(-1/2,1/2).
\end{equation}
Given $a\in[x_1,x_2)$ and $c\in(a,x_2]$, define
\begin{equation}
\label{eq:omega}
\omega(x)\defeq \begin{dcases}
	\frac{c-x}{c-a}, &x\in[a,c],
	\\
	0, &x\in(c,b].
		   \end{dcases}
\end{equation}
Then,
\begin{equation*}
\sum_{\lambda\in\spec(A_{b}^\gamma)}
	\frac{\abs{J_{ab;\omega}(z,\lambda)}}{K_b(\lambda,\lambda)^{1/2}}
	< \infty
\end{equation*}
for every $\gamma\in[0,\pi)$ (barring the exceptional extension), uniformly on compact
subsets of $\C$.
\end{proposition}
\begin{proof}
Set $a\in[x_1,x_2)$ and $c\in(a,x_2]$. Within $[a,c]$
equation \eqref{eq:canonical-system-base} becomes
\[
\begin{aligned}
	- y'_1(x) &= z \left(\tfrac12 - g_0\right) y_2(x)
	\\[1mm]
	  y'_2(x) &= z \left(\tfrac12 + g_0\right) y_1(x)
\end{aligned}
\]
with boundary condition
\(
u(z,a)
\). The solution within this interval is
\[
u(z,x) =
	\begin{pmatrix}
	A\cos(z\kappa x) + B\sin(z\kappa x)
	\\
	\alpha A\sin(z\kappa x) - \alpha B\cos(z\kappa x)
	\end{pmatrix},
\]
where $A$ and $B$ are defined by the boundary condition at $x=a$,
\begin{equation*}
\alpha = \sqrt{\frac{1 + 2 g_0}{1 - 2 g_0}},\quad\text{and}\quad
\kappa = \frac12\sqrt{1 - 4 g_0^2}.
\end{equation*}
As a consequence,
\[
\norm{u(\lambda,\cdot)}_{L^2_H(b)}^2
	\ge \int_{a}^{c} u(\lambda,x)^t H(x) u(\lambda,x) dx
	= \left(\tfrac12 + g_0\right) \left(A^2 + B^2\right)  (c - a).
\]
Also,
\[
\begin{aligned}
\int_{a}^{c} u(z,x)^t \cJ u(\lambda,x) dx
	&= \alpha \left(A^2 + B^2\right) \int_{a}^{c} \sin\left((z-\lambda)\kappa x\right) dx
	\\[1mm]
	&= - \frac{\alpha \left(A^2 + B^2\right)}{(z-\lambda)\kappa }
		\left.\cos\left((z-\lambda)\kappa x \right)\right|_{a}^{c}.
\end{aligned}
\]
Therefore, for sufficiently large $N$,
\begin{equation*}
\sum_{\substack{\lambda\in\sigma(A_{b}^{\gamma}) \\ \abs{\lambda}> N}}
		\frac{\abs{\int_a^c u(z,x)^t \cJ u(\lambda,x) dx}}
		{\abs{z-\lambda}\norm{u(\lambda,\cdot)}_{L^2_H(b)}}
	\le C(\K) \sum_{\lambda\in\sigma(A_b^\gamma)} \frac{1}{1+\lambda^2} < \infty,
\end{equation*}
where the convergence follows from \cite[Thm.~3.4]{remling-book}.
\end{proof}

In what follows, if $I=[x_1,x_2]$, we use the notation $I_\circ$ 
for the half-closed interval of the form $[x_1,x_2)$.

\begin{corollary}
\label{cor:the-corollary}
Let $\cB$ be a symmetric, regular de Branges space.
Denote $\tau(l) = \tau(\cB_l)$.
Suppose $\tau$ is a (non constant) linear function
on a closed interval $I\subset (0,b]$. 
Then, $\cB=\cB_b$ has the $\ell_{\infty}$-oversampling property relative 
to every $\cB_l\in\sub(\cB_a)$ whenever $a\in I_\circ$.
\end{corollary}
\begin{proof}
Let $H\in\cC(b)$ be the coefficient matrix associated with $\cB$ according to
Theorem~\ref{thm:big-theorem}(i). In view of Theorem~\ref{thm:big-theorem}(ii),
we can write
\begin{equation}
\label{eq:our-H}
H(x) = \begin{pmatrix} \frac12 + g(x) & 0 \\[1mm] 0   & \frac12 - g(x)  \end{pmatrix},
	\quad g\in L^1(0,b),\quad g(x)\in[-1/2,1/2].
\end{equation}
Let $I=[x_1,x_2]\subset(0,b]$ be the interval where $\tau$ is a linear function.
It follows that $\det H$ is a non-zero constant in there so $g\equiv g_0 \in(-1/2,1/2)$.
Now apply
Proposition~\ref{prop:oversampling-restricted-diagonal-system-constant}.
\end{proof}

Oscillation theory will be useful in what follows \cite{remscar}.
Assuming $z=\lambda\in\R$, let $R(\lambda,x)>0$ and
$\theta(\lambda,x)$ be absolutely continuous real-valued functions such that
$u(\lambda,x) = R(\lambda,x) \xi_{\theta(\lambda,x)}$.
In terms of these functions, and using the parametrization \eqref{eq:our-H} for $H$ (which
can always be done), the initial-value problem
\begin{equation}
\label{eq:ivp-canonical}
\cJ u'(x)= -\lambda H(x) u(x), \quad u(\lambda,0)
			= \begin{pmatrix} 1 \\ 0 \end{pmatrix},
\end{equation}
becomes
\begin{align}
\theta'(x)
	&= \lambda \left[\tfrac12 + g(x)\cos 2\theta(x)\right],\quad \theta(\lambda,0) = 0,
		\label{eq:ivp-theta}
\\[1mm]
R'(x)
	&= \lambda R(x) g(x) \sin 2\theta(x),\quad R(\lambda,0) = 1.
		\nonumber
\end{align}
Moreover, the boundary condition
\[
u_1(\lambda,b)\sin\gamma - u_2(\lambda,b)\cos\gamma = 0
\]
translates into
\begin{equation}
\label{eq:eigenvalues-on-theta}
\theta(\lambda,b) = \gamma \mod \pi,
\end{equation}
that is, $\lambda\in\R$ is an eigenvalue if and only if there exists $n\in\Z$
such that
\(
\theta(\lambda,b) = \gamma + n\pi.
\)

\begin{lemma}
\label{lem:bound-on-norm}
Suppose there exists an interval $[a,c]\subset(0,b]$ such that $g|_{[a,c]}$
is absolutely continuous 
and $g(x)\in[-1/2+\delta,1/2-\delta]$ for some $\delta>0$ within this interval. Then,
\[
\norm{u(\lambda,\cdot)}_{L^2_H(b)} \ge C_1 R(\lambda,a)
\]
for some positive constant $C_1$ (independent of the spectral parameter).
\end{lemma}
\begin{proof}
Our hypothesis on $g$ implies
\[
\norm{u(\lambda,\cdot)}^2_{L^2_H(b)}
	\ge \int_a^b R(\lambda,x)^2 \left[\tfrac12 + g(x)\cos 2\theta(\lambda,x)\right] dx
	\ge C_0 \int_a^b R(\lambda,x)^2 dx
\]
for some positive constant $C_0$. Since
\begin{equation}
\label{eq:R-is-bounded}
R(\lambda,x)
	= R(\lambda,a) e^{\lambda\int_a^x g(s) \sin 2\theta(\lambda,s)ds},
\end{equation}
it suffices to prove that the exponential factor is uniformly bounded with respect to the
spectral parameter.

In view of \eqref{eq:ivp-theta}, $\theta(\lambda, s)$ is invertible with respect to $s$
for every $\lambda\ne 0$. If $s(\lambda,\theta)$ is this  inverse, then
\[
\int_a^x g(s) \sin 2\theta(\lambda,s)ds
	= \frac{1}{\lambda} \int_{\theta(\lambda,a)}^{\theta(\lambda,x)}
		\frac{2 g(s(\lambda,\theta))\sin 2\theta}{1+ 2g(s(\lambda,\theta))\cos 2\theta} d\theta.
\]
But
\[
\frac{2 g(s(\lambda,\theta))\sin 2\theta}{1+ 2g(s(\lambda,\theta))\cos 2\theta}
	= -\frac12\frac{d}{d\theta}\log\left(1+ 2g(s(\lambda,\theta))\cos 2\theta\right)
		+ \frac{\frac{d}{d\theta}g(s(\lambda,\theta))\cos 2\theta}
		{1+ 2g(s(\lambda,\theta))\cos 2\theta}
\]
and
\[
\frac{d}{d\theta}g(s(\lambda,\theta))
	= \frac{2 g'(s(\lambda,\theta)}{\lambda \left(1+ 2g(s(\lambda,\theta))\cos 2\theta\right)}.
\]
Thus,
\begin{multline*}
\int_a^x g(s) \sin 2\theta(\lambda,s)ds
	= \frac{1}{\lambda}\log
		\left(\frac{1+ 2g(a)\cos 2\theta(\lambda,a)}{1+ 2g(x)\cos 
		2\theta(\lambda,x)}\right)^{1/2}
	\\[1mm]
	+ \frac{1}{\lambda}
		\int_a^x\frac{g'(x) \cos 2\theta(\lambda,x)}{1 + 2g(x)\cos 2\theta(\lambda,x)} dx.
\end{multline*}
From this equality the uniform boundedness of \eqref{eq:R-is-bounded} follows.
\end{proof}

\begin{theorem}
\label{thm:the-theorem}
Let $\cB$ be a symmetric, regular de Branges space.
Denote $\tau(l) = \tau(\cB_l)$. Suppose $\tau$ is strictly increasing
on a closed interval $I\subset (0,b]$ and $\tau'$ is absolutely continuous on $I$. Then,
$\cB=\cB_b$ has the $\ell_p$-oversampling property relative to
every $\cB_l\in\sub(\cB_a)$ for every $p\in(2,\infty)$ and $a\in I_\circ$.
\end{theorem}
\begin{proof}
Let $H\in\cC(b)$ be the coefficient matrix associated with $\cB$ according to
Theorem~\ref{thm:big-theorem}(i). Then,
\[
H(x) = \begin{pmatrix} \frac12 + g(x) & 0 \\[1mm] 0   & \frac12 - g(x)  \end{pmatrix},
	\quad g\in L^1(0,b),\quad g(x)\in[-1/2,1/2].
\]
Let $I=[v_1,v_2]\subset(0,b]$ be the interval where $\tau$ is strictly
increasing and $\tau'$ is absolutely continuous. It follows that $g$ is also
absolutely continuous there and $g(x_0)\in(-1/2,1/2)$ for some $x_0\in(v_1,v_2)$.
Then, given $\delta>0$ small enough, there exists $[x_1,x_2]\subset I$ such that
$g(x)\in[-1/2+\delta, 1/2-\delta]$ for all $x\in[x_1,x_2]$.

Choose $a\in[x_1,x_2)$, $c\in(a,x_2]$ and consider $\omega\in\text{AC}[a,b]$ defined
by \eqref{eq:omega}. Then, given a compact subset $\K\subset\C$, there exists a positive
constant $C_2(\K)$ such that
\begin{equation*}
\abs{\int_a^b u(z,x)^t \cJ u(\lambda,x) \omega'(x) dx}
	\le \int_a^c R(\lambda,x) \abs{u(z,x)^t \cJ  \xi_{\theta(\lambda,x)}} dx
	\le C_1(\K) R(\lambda,a).
\end{equation*}
Therefore, taking into account \eqref{lem:bound-on-norm}, there exists $C(\K)$ 
and $N>0$ such that $\lambda\not\in\K$ as long as $\abs{\lambda}>n$ and, recalling
\eqref{eq:expression-to-bound},
\begin{equation}
\label{eq:last-thing}
\sum_{\substack{\lambda\in\sigma(S_{b}^{\gamma}) \\ \abs{\lambda}> N}}
		\left(\frac{\abs{J_{ab;\omega}(z,\lambda)}}{K_b(\lambda,\lambda)^{1/2}}
		\right)^q
	= C(\K)\sum_{\substack{\lambda\in\sigma(S_{b}^{\gamma}) \\ \abs{\lambda}> N}}
		\frac{1}{\abs{\lambda}^q}.
\end{equation}
But, since $\cB_b$ is regular, the elements of $\sigma(S_{b}^{\gamma})$ obeys condition
(C2) of \cite[Thm.~1.1]{woracek-00}, which in turn implies the convergence of 
\eqref{eq:last-thing} for any $q\in(1,2)$.
\end{proof}


\section{An example}
\label{sec:example}

Consider a canonical system defined on an interval $[0,b]$ by a coefficient matrix $H$ 
that is not trace normalized. Define
\begin{equation}
\label{eq:change-of-variables}
y(x) \defeq \int_0^x \tr H(s) ds.
\end{equation}
If $u(z,x)$ solves
\begin{equation*}
\cJ u'(x) = -z H(x)u(x),\quad x\in(0,b),
\end{equation*}
then $v(z,y)\defeq u(z, x(y))$ solves
\begin{equation*}
\cJ v'(y) = -z H_1(y)v(y),\quad y\in(0,b_1),
\end{equation*}
where
\[
H_1(y) \defeq \frac{1}{\tr H(x(y))} H(x(y))
\]
and $b_1 = y(b)$ \cite[Thm.~1.5]{remling-book}. Clearly, $H_1\in\cC(b_1)$.
The change of variables \eqref{eq:change-of-variables} also induces an isometry 
\[
L^2_H(b) \overset{V}{\to}L^2_{H_1}(b_1)
\]
according to the rule
\[
(Vf)(y) = f(x(y)).
\]
As a consequence, if $\cB(H,b)$ denotes the de Branges space associated with 
the pair $\{b,H\}$, then $\cB(H,b)\equiv\cB(H_1,b_1)$ isometrically. Moreover, if
$\tau_1(s_1)$ is the exponential type of the de Branges subspace $\cB(H_1,s_1)$, then
$\tau_1(s_1) = \tau(s)$ with $s_1=y(s)$, and $\tau(s)$ being the exponential type of
$\cB(H,s)$.

\bigskip

Without further ado, let us consider the de Branges space $\cB(H,b)$ associated
with the canonical system given by
\begin{equation}
\label{eq:matrix-with-linear-entry}
H(x) \defeq \begin{pmatrix}
		1 & 0 \\[1mm] 0 & x
		\end{pmatrix},\quad x\in[0,b],
\end{equation}
and boundary condition $(\begin{smallmatrix}1\\0\end{smallmatrix})$ at $x=0$.
As already discussed, $\cB(H,b)$ is isometrically equal to $\cB(H_1,b_1)$, where
\begin{equation*}
H_1(y) \defeq \begin{pmatrix}
		\frac{1}{\sqrt{1+2y}} & 0 \\[1mm] 0 & \frac{\sqrt{1+2y}-1}{\sqrt{1+2y}}
		\end{pmatrix},\quad y\in[0,b_1],\quad b_1 = b + \tfrac12b^2,
\end{equation*}
and the same boundary condition at $x=0$. Notice that $\tau(s) = \tfrac32 s^{3/2}$.

For the coefficient matrix \eqref{eq:matrix-with-linear-entry}, equation 
\eqref{eq:canonical-system-base} becomes
\begin{equation}
\label{eq:example-with-airy-functions}
\begin{aligned}
	- u'_1(x) &= z x u_2(x),
	\\[1mm]
	  u'_2(x) &= z u_1(x).
\end{aligned}
\end{equation}
Clearly, this is equivalent to solving the second order differential equation
\[
u''_2(x) = -z^2 xu_2(x),\quad x\in[0,b],\quad z\in\C,
\]
with boundary conditions $u_2(0) = 0$ and $u_2'(0) = 1$.
Thus, the solution to \eqref{eq:example-with-airy-functions} is
\[
\begin{aligned}
u_1(z,x)
	&= \wi'(-z^{2/3}x),
	\\[1mm]
u_2(z,x)
	&= -z^{1/3} \wi(-z^{2/3}x),
\end{aligned}
\]
where the function $\wi$ is certain linear combination of the Airy functions
discussed in the Appendix.
In view of \eqref{eq:si-taylor}, it is clear that $u(\cdot,x)$ is a real entire 
$\C^2$-function.

\bigskip

Let $A_b^\gamma$ be the associated self-adjoint operator whose boundary
condition at $x=b$ is defined by $\gamma\in[0,\pi)$. Its spectrum is
\[
\sigma(A_b^\gamma)
	= \left\{\lambda\in\R :
		\wi'(-\lambda^{2/3}b)\sin\gamma
		+ \lambda^{1/3}\wi(-\lambda^{2/3}b)\cos\gamma = 0\right\}.
\]
In particular, the Dirichlet-type spectrum is
\[
\sigma(A_b^0)
	= \left\{\lambda\in\R\setminus\{0\} :
		\wi(-\lambda^{2/3}b) = 0\right\}\cup\{0\};
\]
notice that it is distributed symmetrically around $0$.
Eigenvalues (which are all simple) will be arranged according to increasing values, that is,
\[
\sigma(A_b^\gamma) = \left\{\lambda_n(\gamma)\right\}_{n\in\Z},
\quad\text{where}\quad
\lambda_n(\gamma) < \lambda_{n+1}(\gamma).
\]
Let $y_n$ be the zeros of the function $\wi(-x)$. Then, the Dirichlet-type spectrum obeys
$\lambda_n(0)^{2/3}b = y_n$ so in particular $\lambda_0(0) = 0$ and, in view of \eqref{eq:y_n},
\[
\lambda_{\pm n}(0)
	= \pm \tfrac32\frac{\pi}{b^{3/2}}(n+\tfrac{1}{12})\left[1 + \cO(n^{-2})\right]
\quad
(n\to\infty).
\]
For $\gamma\in(0,\pi)$, define $\beta\defeq(\cot\gamma)/b^{1/2}$. Let $w(\lambda,x)$ be
the function defined by \eqref{eq:w}. We note the following: 
If $\lambda > 0$, then
\[
\lambda\in\sigma(A_b^\gamma)
\quad\text{if and only if}\quad
w(\beta,\lambda^{2/3}b) = 0,
\]
while if 
$\lambda < 0$, then
\[
\lambda\in\sigma(A_b^\gamma)
\quad\text{if and only if}\quad
w(-\beta,(-\lambda)^{2/3}b) = 0.
\]
Therefore $\lambda_0(0)<\lambda_0(\gamma)<\lambda_1(0)$ and, according to \eqref{eq:x_n},
\[
\lambda_{\pm n}(\gamma)
	= \pm \tfrac32\frac{\pi}{b^{3/2}}(n+\tfrac{7}{12})
	\left[1 \mp \frac{\arctan\beta}{\pi}(n+\tfrac{7}{12})^{-1} + \cO(n^{-2})\right]
\quad
(n\to\infty).
\]

\bigskip

Now, let us take a look at the denominator in \eqref{eq:expression-to-bound}.
Henceforth, assuming $\lambda\ne 0$, we have
\begin{align*}
\norm{u(\lambda,\cdot)}^2_{L^2_H(b)}
	&= \lambda^{-2/3} \left[\int^0_{-\lambda^{2/3}b} \wi'(y)^2 dy
		- \int^0_{-\lambda^{2/3}b} y \wi(y)^2 dy\right]
	\\[1mm]
	&= \frac{\lambda^{-2/3}}{3}
		\left[\wi(y)\wi'(y) + 2y\wi'(y)^2 - 2y^2\wi(y)^2\right]^0_{-\lambda^{2/3}b}
	\\[1mm]
	&= \frac{\lambda^{-2/3}}{3}
		\left[2(\lambda^{2/3}b)^2\wi(-\lambda^{2/3}b)^2\right.
	\\[1mm]
	&\quad\left. +\; 2\lambda^{2/3}b\wi'(-\lambda^{2/3}b)^2
			- \wi(-\lambda^{2/3}b)\wi'(-\lambda^{2/3}b)\right]
\end{align*}
(see \cite[\S 3.2.2]{vallee} for details). In particular,
\begin{align*}
\norm{u(\lambda_{\pm n}(0),\cdot)}^2_{L^2_H(b)}
	&= \tfrac23 b \wi'(-y_n)^2
	 =\tfrac23 b C_0^2(\tfrac32\pi n)^{1/3}\left[1 + \cO(n^{-1})\right].
\end{align*}
Similarly,
\begin{align}
\norm{u(\lambda_{\pm n}(\tfrac{\pi}{2}),\cdot)}^2_{L^2_H(b)}
	&= \tfrac23 b x_n(0) \wi(-x_n(0))^2
	 =\tfrac23 b C_0^2(\tfrac32\pi n)^{1/3}\left[1 + \cO(n^{-1})\right].
	\label{eq:norm-airy-example}
\end{align}
With some extra work it can be shown that \eqref{eq:norm-airy-example} holds
true for all $\gamma\in[0,\pi)$. As a consequence, we obtain the following uniform
characterization for the spaces of samples,
\[
\ell_p(\cB,\gamma)
	= \left\{\{\alpha_n\}_{n\in\Z}\subset\C :
		\left\{\frac{\alpha_n}{(1+\abs{n})^{1/6}}\right\}\in \ell_p\right\}
\]
for every $\gamma\in[0,\pi)$ and $p\in(2,\infty]$.

\bigskip

Finally, let us look at the numerator in \eqref{eq:expression-to-bound},
\begin{align}
\int_a^b &u(z,x)^t \cJ u(\lambda,x) dx\nonumber
	\\[1mm]
	&= \int_a^b u_2(z,x)u_1(\lambda,x) dx - \int_a^b u_1(z,x)u_2(\lambda,x) dx\nonumber
	\\[1mm]
	&= - \int_a^b z^{1/3}\wi(-z^{2/3}x)\wi'(-\lambda^{2/3}x) dx
		+ \int_a^b \lambda^{1/3}\wi(-\lambda^{2/3}x)\wi'(-z^{2/3}x) dx. \label{eq:example}
\end{align}
The first term above can be written as follows,
\begin{multline*}
\text{1st term}
	= - C_0 \lambda^{1/6}
		\int_a^b x^{1/4}\cos(\tfrac23\lambda x^{3/2}-\tfrac{\pi}{12}) z^{1/3}\wi(-z^{2/3}x) dx
		\\[1mm]
	  - \int_a^b z^{1/3}\wi(-z^{2/3}x)g(\lambda^{2/3}x) dx,
\end{multline*}
where $g(x) = \cO(x^{-5/4})$ as $x\to\infty$ (see \eqref{eq:u2-with-airy}). An integration
by parts implies
that, given any compact subset $\K\subset\C$, there exists $C(\K)>0$ such that
\[
\abs{\text{1st term}}
	\le \frac{C(\K)}{\abs{\lambda}^{5/6}}
\]
for all $z\in\K$. Since an analogous reasoning can be applied to the second term in 
\eqref{eq:example}, we obtain
\[
\abs{\int_a^b u(z,x)^t \cJ u(\lambda,x) dx}
	\le \frac{2C(\K)}{\abs{\lambda}^{5/6}}.
\]
All in all, given $\gamma\in[0,\pi)$ and any compact subset $\K\subset\C$, there
exists $n_0\in\N$ such that 
$\lambda_n(\gamma)\not\in\K$ for all $\abs{n}\ne n_0$ and
\[
\sum_{\abs{n}\ge n_0}
		\frac{\abs{\int_a^b u(z,x)^t \cJ u(\lambda_n(\gamma),x) dx}}
		{\abs{\lambda_n(\gamma) - z}\norm{u(\lambda_n(\gamma),\cdot)}_{L^2_H(b)}}
	\le C(\K) \sum_{\abs{n}\ge n_0} \frac{1}{n^2}
\]
for all $z\in\K$. That is, $\cB(H,b)$ has the $\ell_\infty$-oversampling property
relative to every de Branges subspace of it.


\section*{Acknowledgments}

This research is based upon work supported by Universidad Nacional del Sur (Argentina)
under grant PGI 24/L117. J.\,H.\,T. thanks IIMAS-UNAM for their kind
hospitality. L.\,O.\,S. is supported by CONACyT CF-2019 \textnumero\,304005.


\appendix

\section{Appendix}

Let us define
\begin{equation*}
\wi(z) \defeq \pi \left[\ai(0)\bi(z) - \bi(0)\ai(z)\right],
\end{equation*}
where $\ai$ and $\bi$ are the Airy functions of the first and second kind. This real
entire function is the solution to the Airy differential equation that obeys the
initial conditions $\wi(0) = 0$
and $\wi'(0) = 1$. Using e.g. \cite[\S 9.2 and \S 9.4]{nist}, one can verify that
\begin{equation}
\label{eq:si-taylor}
\wi(z)  = \sum_{k=0}^\infty\frac{(-1)^{k}c(k)}{(3k+1)!}z^{3k+1},
	\quad\text{where}\quad
c(k) = \prod_{l=0}^{k-1}(2+3l).
\end{equation}
$\wi$ and $\wi'$ admit various asymptotic expansions as $z\to\infty$ that are uniform in
different annular sectors of the complex plane. We only mention those to be used in this work,
namely,
\begin{align*}
\wi(-z)
	&= C_0 z^{-1/4}\left[-\sin(\zeta - \phi_0)
			P(\zeta)
			+\cos(\zeta - \phi_0)
						Q(\zeta)\right]
\intertext{and}
\wi'(-z)
	&= C_0 z^{1/4}\left[\cos(\zeta - \phi_0)
			R(\zeta)
			+\sin(\zeta - \phi_0)
						S(\zeta)\right],
\end{align*}
as $z\to\infty$ with $\abs{\arg(z)}\le 2\pi/3-\delta$, where $\delta$ is a arbitrarily small 
positive number.
Here
\[
\zeta = \tfrac23 z^{3/2},
\quad
C_0 = \sqrt{\pi(\ai(0)^2 + \bi(0)^2)}
	= \frac{2\sqrt{\pi}}{3^{2/3}\Gamma(2/3)}
\quad\text{and}\quad
\phi_0 = \frac{\pi}{4} - \arctan\!\left(\frac{\bi(0)}{\ai(0)}\right)
		= \frac{\pi}{12};
\]
see \cite{nist} for details and \cite{fabijonas} for the definition of the asymptotic expansions
$P$, $Q$, $R$ and $S$.
In particular,
\begin{align}
\wi(-x)
	&= - C_0 x^{-1/4}\sin(\tfrac23 x^{3/2} - \tfrac{\pi}{12})
		+ \cO(x^{-7/4}) \nonumber
\intertext{and}
\wi'(-x)
	&= C_0 x^{1/4}\cos(\tfrac23 x^{3/2} - \tfrac{\pi}{12})
		+ \cO(x^{-5/4}), \label{eq:u2-with-airy}
\end{align}
as $x\to\infty$ along the real line.
Also, let us define $w:\R\times\R_+\to\R$ by the rule
\begin{equation}
\label{eq:w}
w(\beta, x) \defeq \wi'(-x) + \beta x^{1/2}\wi(-x).
\end{equation}
Clearly,
\[
w(\beta, x)
	= C_0 x^{1/4}\left[\cos(\tfrac23 x^{3/2} - \tfrac{\pi}{12})
		- \beta\sin(\tfrac23 x^{3/2} - \tfrac{\pi}{12})\right]
		+ \cO(x^{-5/4})
\quad
(x\to\infty).
\]

Let $\{y_n\}_{n=0}^\infty$ and $\{x_n(\beta)\}_{n=0}^\infty$ be the sequences of zeros of 
$\wi(-x)$ and $w(\beta,x)$ respectively, arranged according to increasing value. 
Note that $y_0=0$ while $x_0(\beta)>0$ for all $\beta\in\R$. Using standard techniques for
inversion of asymptotic expansions, along with the phase principle for enumeration of 
the resulting sequence of zeros (see \cite{fabijonas}), one can prove that
\begin{align}
y_n &= \left(\tfrac32\pi(n+\tfrac{1}{12})\right)^{2/3}
		\left[1 + \cO(n^{-2})\right]
	\label{eq:y_n}
\intertext{and}
x_n(\beta)
	&= \left(\tfrac32\pi(n+\tfrac{7}{12})\right)^{2/3}
		\left[1 - \arctan\beta\left(\tfrac32\pi(n+\tfrac{7}{12})\right)^{-1} 
		+ \cO(n^{-2})\right],\label{eq:x_n}
\end{align}
as $n\to\infty$.


\end{document}